\documentclass[twoside,12pt]{article}
\usepackage{amsmath,amsthm,amssymb,amscd}
\usepackage[matrix,arrow,curve]{xy}
\usepackage[pdftex]{graphicx}
\title{Angle contraction between geodesics}

\author{Nikolai A. Krylov and Edwin L. Rogers}

\date {}

\newtheorem{theorem}{Theorem}
\newtheorem{corollary}{Corollary}
\newtheorem{lemma}{Lemma}

\def\real            {\mathbb R}

\def\lra            {\longrightarrow}

\def\btr            {\bigtriangleup}

\begin{document}

\maketitle

\parskip=5mm

\begin{abstract}
{We consider here a generalization of a well known discrete dynamical system produced by the
bisection of reflection angles that are constructed recursively between two lines in the Euclidean plane.
It is shown that similar properties of such systems are observed when the plane is replaced by a regular
surface in $\real^3$ and lines are replaced by geodesics. An application of our results to the
classification of points on the surface as elliptic, hyperbolic or parabolic is also presented.}
\end{abstract}


\section{Introduction}

Let us consider the following simple problem. In the Euclidean plane
take two rays starting at point $V$ and forming an acute angle of
measure $\mu$. Denote the rays by $L_A$ and $L_B$ and take an
arbitrary transversal segment $A_1B_1$ with $A_1\in L_A$, and
$B_1\in L_B$. Keep constructing further segments with respect to the
following rule: To create a new transversal, take the angle between
the most recently constructed transversal and the corresponding ray,
and let the new transversal be the bisector of this angle. In
particular, if we have as our last transversal $B_{k-1}A_k$, we take
the bisector of $\angle B_{k-1}A_k V$, denote its intersection with
line $L_B$ by $B_k$ and use the transversal $A_kB_k$ for the next
step (see Figure 1). Clearly, this process creates not only two
sequences of points $\{A_k\} \in L_A$, $\{B_k\} \in L_B$ but also
two sequences of angles $\{\alpha_k\} $, and $\{\beta_k\} $, where
$\alpha_k=\angle B_kA_kV$, and $\beta_k=\angle A_{k+1}B_kV$.

\begin{figure*}[h]
\includegraphics[height=75mm,width=130mm]{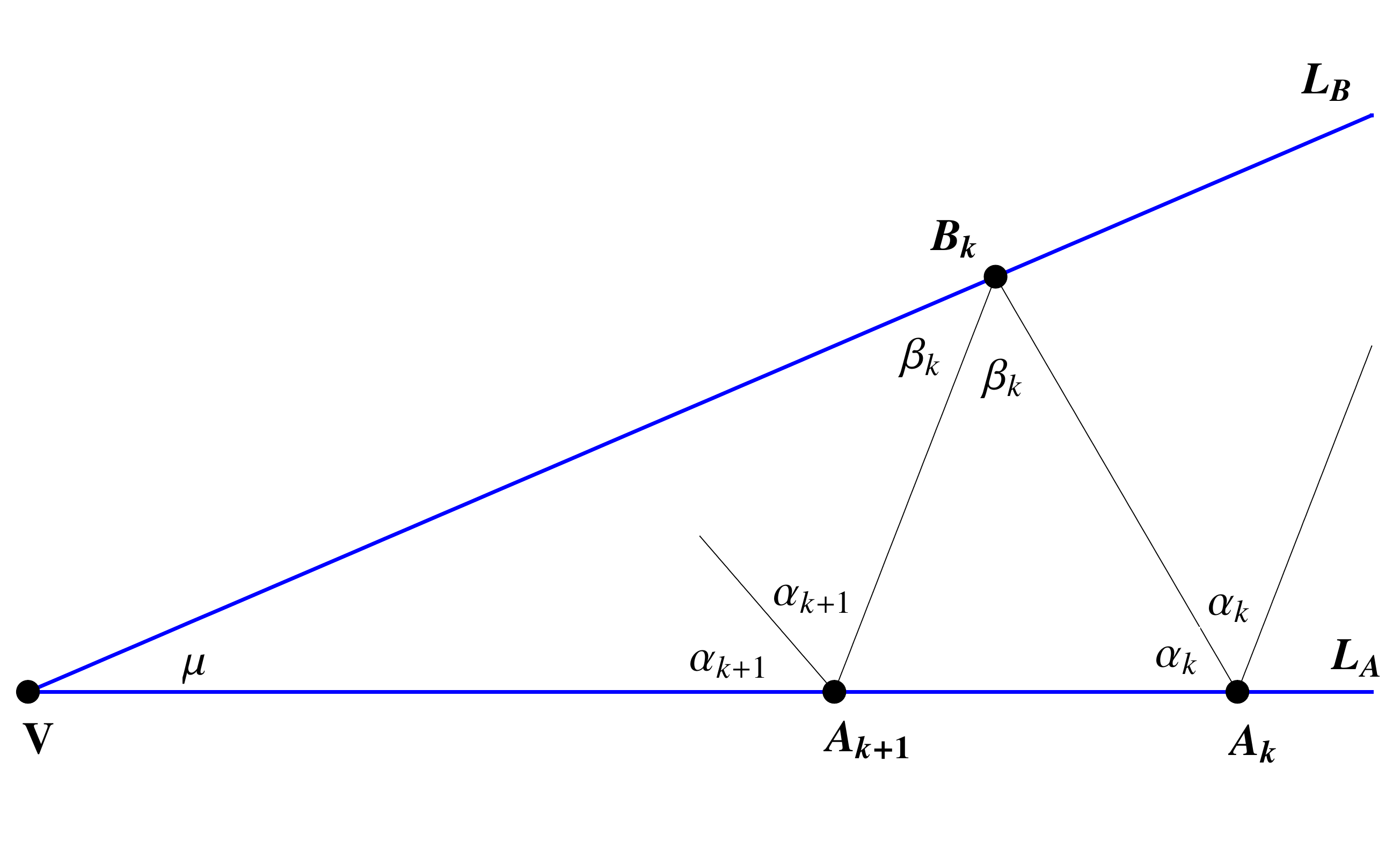}
\caption{}
\end{figure*}

\noindent By expressing each sequence of angles as a first order difference equation, it is easy to show that
$$
\lim_{k\to\infty} \alpha_k = \lim_{k\to\infty} \beta_k = \frac{\pi -
\mu}{3}
$$

\noindent irrespective of the choice of the initial transversal segment.  We show below that this property holds
under more general circumstances.  Specifically, we replace the plane with a smooth surface in
$\real^3$, where our triangles will be geodesic triangles,  and drop the condition that the angles are created by bisection.

\section{A generalization of the problem}

Let $S\subset \real^3$ be a regular surface locally parametrized by a
differentiable vector function ${\bf r}(u,v): U\lra \real^3$, with
$U$ being open in $\real^2$. All curves we consider below will be
{\it regular parametrized curves}, that is, differentiable vector functions ${\bf c}:[a,b] \lra \real^3$ with
$\dot{\bf c}(t) = {\rm d}{\bf c}(t)/{\rm d}t \ne 0$ for all values
of the parameter $t\in (a,b)$. By a {\it geodesic} on $S$ we will mean a unit
speed regular curve ${\bf c}(s)$ (parametrized by the arc length)
on $S$ such that the second derivative ${\bf c}''(s)$
is the zero vector or perpendicular to the surface.

To avoid the cases when two distinct points on $S$ can be joined by
different geodesics, we will consider only ``small" triangles formed
by two geodesics intersecting at a point $V$ and a third geodesic
intersecting the first two transversally at points $A$ and $B$. By a
``small" triangle we will mean one contained in an open subset of
$S$, which is a normal neighborhood of all its points. Proof of the
existence of such a neighborhood for each point $p\in S$ can be
found in \S 4-7 of \cite{Carmo}.

If $\gamma_v:[0,b] \lra S, \gamma_v(0)=P, \gamma_v'(0)=v$ and $\gamma_w:[0,c] \lra S, \gamma_w(0)=P$,
$\gamma_w'(0)=w$ are two geodesics at $P$, the angle between these two geodesics is defined to be the angle
between $v$ and $w$.  If $X$ is on $\gamma_v$ and $Y$ is on $\gamma_w$, we will denote this angle either by
$\angle XPY$ or by $\angle YPX$.  Assume now that we are given a ``small" triangle $\btr VAB$ on a
regular surface $S\subset \real^3$ and denote $\angle BVA$ by $\mu$, and
assume that $\mu < \pi$.

As was the case in the Euclidean plane,  we can construct first order difference equations for the angles created
by bisection, where now our transversals will be geodesic segments.  Let $A_1=A,~B_1=B$ and $\alpha_1:=
\angle VA_1B_1$. Assume also that we are given two continuous
functions $p: S\to \real^+$ and $q: S\to \real^+$ taking only positive values. To construct the point $A_2$ on the geodesic
segment $VA_1$, we consider the geodesic $\gamma_{\beta_1}:[0,1]\to S$ such
that $\gamma_{\beta_1}(0)=B_1$ and the angle between $\gamma'_{\beta_1}(0)\in
T_{B_1}(S)$ and the geodesic segment $B_1A_1$ is
$$
\beta_1= \frac{\angle VB_1A_1}{1+q(B_1)}.
$$

Geodesic $\gamma_{\beta_1}([0,a))$ (with $a$ large enough) will
intersect transversally the segment $VA_1$ at a point, which we
denote by $A_2$. Then we use function $p:S\to \real^+$ and divide
angle $\angle B_1A_2V$ by the geodesic $\gamma_{\alpha_2}:[0,1]\to S$ to
create the point $B_2$ on segment $VB_1$, where $\gamma_{\alpha_2}(0)=A_2$ and
the angle between $\gamma'_{\alpha_2}(0)\in T_{A_2}(S)$ and segment $A_2V$ is
defined by
$$
\alpha_2= \frac{\angle VA_2B_1}{1+p(A_2)}.
$$

Point $B_2$ is defined as the intersection of $VB_1$ with geodesic
$\gamma_{\alpha_2}([0,a))$ (see Figure 2, where we abbreviated $p(A_m)$ and
$q(B_n)$ as $p_m$ and $q_n$ respectively).

\begin{figure*}[h]
\includegraphics[height=75mm,width=140mm]{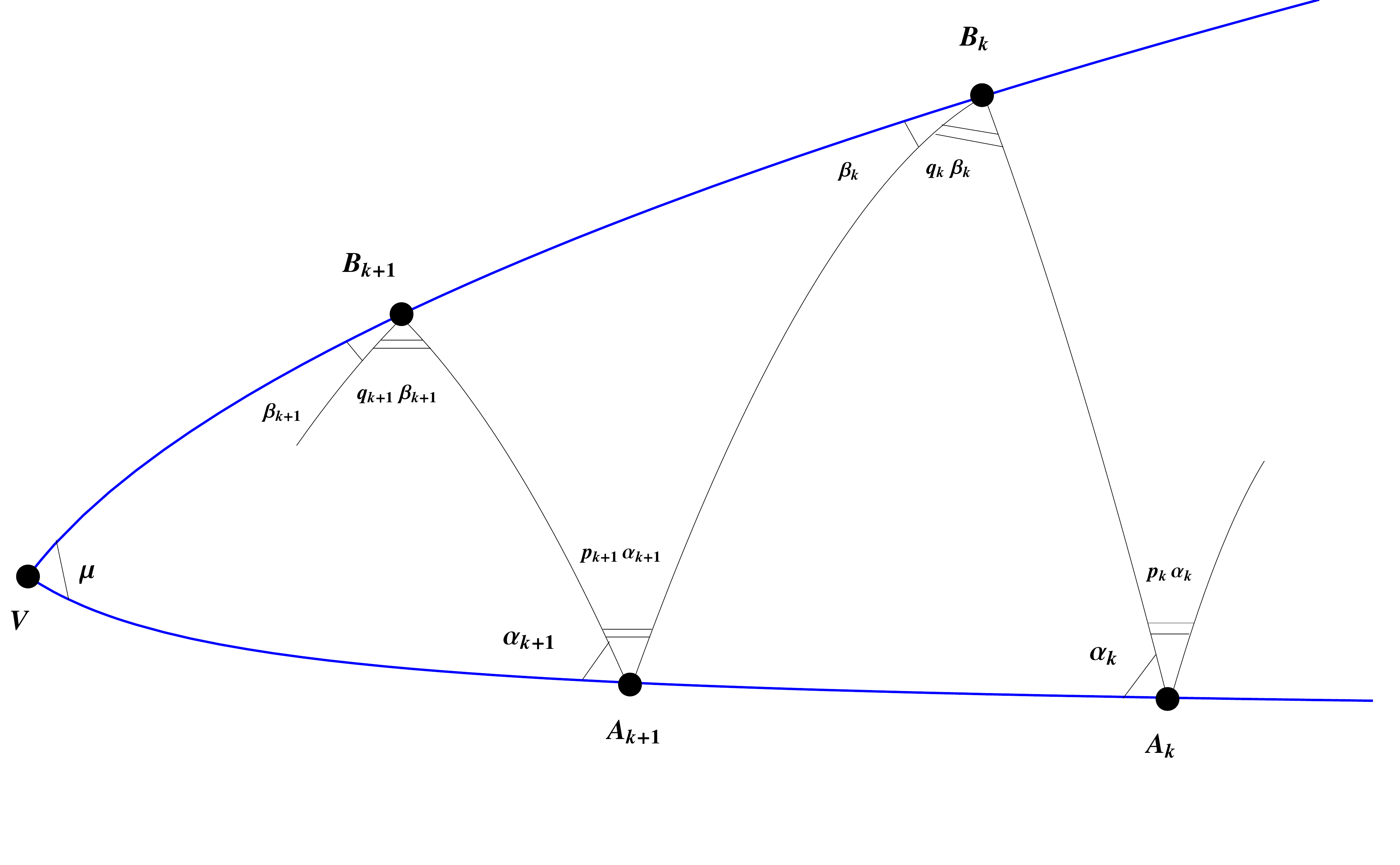}
\caption{}
\end{figure*}

\noindent Then we iterate this procedure to construct two sequences
of points $\{A_k\}_{k=1}$, $\{B_k\}_{k=1}$, and angles
$$
\alpha_{k+1} =\frac{\angle VA_{k+1}B_k}{1+p(A_{k+1})},~~~~~{\rm
and}~~~~~\beta_k =\frac{\angle VB_kA_k}{1+q(B_k)}
$$

Here arises a natural question. Do the sequences
$\{\alpha_n\}$ and $\{\beta_n\} $ converge, and if so,
to what values? In the next section we will prove the following
\begin{theorem}
For an arbitrary small triangle $\btr VAB$ on a regular surface $S$
in $\real^3$ with the angle $\angle AVB = \mu < \pi$ and two
continuous positive functions $p:S\lra \real^+$ and $q:S\lra
\real^+$, we have
$$
\lim_{n\to\infty} \alpha_n = \frac{q(V)\cdot(\pi -
\mu)}{p(V)+q(V)+p(V)\cdot q(V)}
$$ and
$$
\lim_{n\to\infty} \beta_n = \frac{p(V)\cdot(\pi -
\mu)}{p(V)+q(V)+p(V)\cdot q(V)}
$$
\end{theorem}

~

\begin{corollary} If, in addition to the conditions of the theorem we
assume that functions $p$ and $q$ are equal, then
$$
\lim_{n\to\infty} \alpha_n = \lim_{n\to\infty} \beta_n = \frac{\pi -
\mu}{2+p(V)}
$$
\end{corollary}

~

Recall that a point $x\in S$ is called elliptic when the Gauss curvature $K(x)>0$, hyperbolic when $K(x)<0$ and parabolic when
$K(x)=0$. Our next corollary shows that there exist functions $p(x)$ and $q(x)$ so that the limits of our sequences of angles
will distinguish the type of the point $V$.

\begin{corollary}
Let $k_1(x)$ and $k_2(x)$ denote two principal curvatures at $x\in
S$. If
$$
p(x)= 1+ |K(x)\cdot(k_1(x)+k_2(x))|~~~{\rm and}~~~q(x)= 1+
|K(x)|\cdot(|k_1(x)|+|k_2(x)|)
$$
then the pairs $(\lim\alpha_n,~\lim\beta_n)$ are different for
different type of points.
\end{corollary}
\begin{proof}
Straightforward computations show that in the case of a parabolic,
elliptic, and hyperbolic points correspondingly, we have
$$
~~\lim\alpha_n=\lim\beta_n =
\frac{\pi-\mu}{3},~~\lim\alpha_n=\lim\beta_n < \frac{\pi-\mu}{3},
~~\lim\alpha_n\neq \lim\beta_n.
$$
\end{proof}

\section{Proof of the theorem}

The proof of  Theorem 1 relies on a generalization of the Banach
contraction mapping principle and a special case of the local
version of the Gauss-Bonnet Theorem. Let us first recall the
contraction principle.

\noindent {\bf The Contraction Theorem} Consider a complete metric
space $(M,{\rm dist})$. Let $T: M\to M$ be a contraction mapping
with Lipschitz constant $k\in (0,1)$, and suppose $\alpha \in M$ is
the fixed point of $T$. Let $\{\varepsilon_n\}$ be a sequence
of positive numbers converging to zero, and suppose
$\{\alpha_n\}\subseteq M$ satisfies:
$$
{\rm dist}(\alpha_{n+1},~T(\alpha_n))\leq \varepsilon_n.
$$ Then the sequence $\{\alpha_n\}$ converges to
$\alpha$.
\noindent We refer the reader to chapter 3 of \cite{Khamsi} for a
proof of this theorem.

\begin{proof}[Proof of Theorem 1]
Let's apply the Gauss-Bonnet Theorem to geodesic triangles
$A_kB_kA_{k+1}$ and $A_kB_kV$.  Using notations from figure 2, one
will get the following:
$$
\iint\limits_{A_kB_kA_{k+1}} K~dS = \alpha_k + q(B_k)\cdot\beta_k +
(\pi-(1+p(A_{k+1}))\cdot\alpha_{k+1}) - \pi,
$$
which is equivalent to the formula
\begin{equation}\label{aba}
\alpha_{k+1} = \frac{\alpha_k}{1+p(A_{k+1})} +
\frac{q(B_k)\cdot\beta_k}{1+p(A_{k+1})} -
\frac{1}{1+p(A_{k+1})}\cdot \iint\limits_{A_kB_kA_{k+1}} K~dS,
\end{equation} and similarly
$$
\iint\limits_{A_kB_kV} K~dS = \alpha_k + (1+q(B_k))\cdot\beta_k +
\mu - \pi,
$$
which gives $\beta$ in terms of $\alpha$ and $\mu$:
\begin{equation}\label{abo}
\beta_k = \frac{\pi- \mu}{1+q(B_k)} - \frac{\alpha_k}{1+q(B_k)} +
\frac{1}{1+q(B_k)}\cdot \iint\limits_{A_kB_kV} K~dS.
\end{equation}

\noindent Substituting (\ref{abo}) into (\ref{aba}) one easily
obtains

%\begin{equation}\label{iterates}
\(
\alpha_{k+1} =\\ ~ \\
\frac{\alpha_k + q(B_k)\cdot(\pi-\mu)}{(1+p(A_{k+1}))\cdot(1+q(B_k))} +
\frac{q(B_k)}{(1+p(A_{k+1}))\cdot(1+q(B_k))}\cdot
\iint\limits_{A_kB_kV} K~dS - \frac{1}{1+p(A_{k+1})}\cdot
\iint\limits_{A_kB_kA_{k+1}} K~dS
\)
%\end{equation}

\noindent Since $p(X)>0$ and $q(X)>0$ for an arbitrary point $X\in
S$, the map
$$
T(\varphi) : =  \frac{1}{(1+p(V))\cdot(1+q(V))}\cdot\varphi +
\frac{q(V)\cdot(\pi-\mu)}{(1+p(V))\cdot(1+q(V))}
$$ will be a contraction with Lipschitz constant $1/[(1+p(V))\cdot(1+q(V))] \in (0,1)$.
Its fixed point $\alpha$ can be found by a straight forward
computation
$$
\alpha=\frac{q(V)\cdot(\pi - \mu)}{p(V)+q(V)+p(V)\cdot q(V)}.
$$
To apply the contraction theorem for this map $T$, we set $M=\real$
with the standard distance function ${\rm dist}(x,y):=|x-y|$. Then
it follows  that \\ ~ \\
$\varepsilon_k := |\alpha_{k+1} - T(\alpha_k)| =
|\frac{\alpha_k+q(B_k)\cdot(\pi-\mu)}{(1+p(A_{k+1}))\cdot(1+q(B_k))}
- T(\alpha_k)+
\frac{q(B_k)}{(1+p(A_{k+1}))\cdot(1+q(B_k))}\cdot
\iint\limits_{A_kB_kV} K dS -\\ \frac{1}{1+p(A_{k+1})}\cdot
\iint\limits_{A_kB_kA_{k+1}} K dS ~|
$

Thus, to prove our theorem for the sequence $\{\alpha_n\}$, it
is enough to show that $\{\varepsilon_k\} \to 0$ when
$k\to\infty$. For this we will show that the difference of two
fractions and each of the double integrals in the formula for
$\varepsilon_k$ converge to zero. The Gaussian curvature of a
regular surface is a smooth function, and hence it is universally
bounded on a small triangle. Since the union of all triangles
$\btr A_kB_kA_{k+1}$ is contained in the small triangle $\btr VAB$, we must
have
$$
\sum\limits_{k=1}^{\infty}~ \iint\limits_{A_kB_kA_{k+1}} \left |K\right |~dS <
\iint\limits_{ABV}  \left |K\right |~dS <\infty,
$$ and hence the general term of the series converges to zero.
As for the integrals over the triangles $A_kB_kV$, we first notice
that both sequences $\{A_k\}$ and $\{B_k\}$ converge to the vertex
$V$ (see Lemma 1 below). Therefore length of each of three
geodesics $L(A_k,V)$, $L(B_k,V)$ and $L(A_k,B_k)$ approaches zero,
which implies that the regions in $\real^2$ corresponding to the
triangles $\btr A_kB_kV$ get shrunk to a point. Hence

$$
 \iint\limits_{A_kB_kV} K~dS  \stackrel{k\to\infty}{\lra} 0.
$$
Convergence to zero of
$$
\frac{\alpha_k+q(B_k)\cdot(\pi-\mu)}{(1+p(A_{k+1}))\cdot(1+q(B_k))}
- T(\alpha_k)
$$
follows from the lemma since both functions $p ,
q:S\lra \real^+$ are continuous. Thus, $\varepsilon_k \lra 0$, when
$k\to\infty$, and hence our result for $\{\alpha_k\}$. To finish the
proof take the limit as $k\to\infty$ of each side in formula
($\ref{abo}$) to obtain the desired answer for the sequence
$\{\beta_k\}$.
\end{proof}

\begin{lemma}
$
\lim\limits_{k\to\infty} B_k = \lim\limits_{k\to\infty} A_k = V
$
\end{lemma}
\begin{proof}
Suppose we are given a geodesic segment $PQ\subset S$. As above, we
will designate here the length of such a segment by $L(P,Q)$. Since
our triangle $\btr B_1VA_1$ is a small one, it is easy to see that
the sequence of points $\{B_k\}$ converges to a point, $P$, on
the geodesic segment $B_1V$. Assume that $P \not= V$. Then $\forall\varepsilon
>0$, sufficiently small, $~\exists N$, a large enough natural number, such that $\forall
k\geq N$, $B_k$ will be in a neighborhood of radius
$\varepsilon$ centered at $P$ and contained in $exp_p(D_{\varepsilon})$.  Further,
we assume that $\varepsilon$ is so small that
$exp_p(D_{\varepsilon})$ has no points in common with the segment
$A_1V$. This implies that $\exists  c>0$ such that for all large
enough $n,~m$ we also have $L(A_n,B_m)>c$. Consider a regular
parametrization $\gamma(t):~[0,\theta]\lra B_1V$ of the geodesic
$B_1V$ such that $\gamma(0) = B_1$, $\gamma(\theta) = V$ and $ {\rm
d}\gamma(t)/{\rm d}t \ne 0,~ \forall t\in (0,\theta).$ Let us denote
by $t_k\in[0,\theta]$, $t_1=0$, the parameter value corresponding to the point
$B_k\in B_1V$, i.e. $\gamma(t_k) = B_k$.  Since
the sequence $\{B_k\}$ converges and $\gamma(t)$ is smooth,
the corresponding sequence of parameters $\{t_k\}$ will
converge as well.

As the next step, we use the exponential map to show that
existence of the limit point $P$ with $L(P,V)>0$ implies the
convergence to zero of both sequences $\{\alpha_k\}$ and
$\{\beta_k\}$. Assume that $n$ is large enough and recall that
$exp_{A_n}:~T_{A_n}S \lra S$ is a diffeomorphism locally and its
image contains our small triangle $\btr ABV$. In particular, we can
consider the preimage of the two geodesics segments  $B_nA_n$, and
$A_nB_{n+1}$, which will be line segments in $T_{A_n}S$. We
denote those lines by $(\tilde{B}_nA_n)$ and $(A_n\tilde{B}_{n+1})$
respectively(i.e. $exp^{-1}_{A_n}(B_n) = \tilde{B}_n$). Since
$exp_{A_k}$ preserves the angles and lengths of rays through the
point $A_k$, we have that
$$
\angle \tilde{B}_nA_n\tilde{B}_{n+1} = \angle B_nA_nB_{n+1} =
p(A_n)\cdot\alpha_n ~~~{\rm
and}~~~|A_n\tilde{B}_{n+i}|=L(A_n,B_{n+i})> c,
$$
where $i\in\{0,1\}$. We also have the following inequalities:
$$
L(\tilde{B}_{n+1},\tilde{B}_n)\geq {\rm
dist}[\tilde{B}_{n+1},(A_n\tilde{B}_n)]> c\cdot{\rm
sin}(p(A_n)\cdot\alpha_n),
$$
where $L(\tilde{B}_{n+1},\tilde{B}_n)$ denotes the length of the
curve $exp^{-1}_{A_n}[(B_{n+1}B_n)]\subset T_{A_n}S$ and can be
computed by the integral
$$
L(\tilde{B}_{n+1},\tilde{B}_n) = \int\limits^{t_{n+1}}_{t_n}
|(exp^{-1}_{A_n} \circ \gamma)'(u)|~du.
$$
Since the composition $exp^{-1}_{A_n} \circ \gamma(t)$ is smooth,
there exists a positive constant $K$ such that for all large enough
$n$, we have
$$
K\cdot|t_{n+1} - t_n|\geq L(\tilde{B}_{n+1},\tilde{B}_n)> c\cdot{\rm
sin}(p(A_n)\cdot\alpha_n),
$$ which implies that the sequence of angles $\{\alpha_n\}$ converges
to zero. Recalling (\ref{abo}) we conclude that the sequence $\{\beta_n\}$ converges to zero
as well.

Now let's look at the triangle $\btr A_kB_kV$ and consider the
geodesic through $A_k$ and a point $Z$ on the geodesic segment
$VB_k$, including both end points. Since $Z$ belongs to a closed
and bounded curve, there is such geodesic $A_kZ$ of maximal length.
Let's denote this length by $L_k$, that is
$$
L_k:= \max\limits_{Z\in B_kV}\left \{L(A_k,Z)\right \}.
$$
Now we take the circular sector of radius $L_k$, denoted by $\btr
S(A_k,L_k)$,  centered at $A_k$
and bounded between two geodesics $A_kV$ and $A_kB_k$ with angle
$\alpha_k$ at $A_k$.  Clearly, $\btr A_kB_kV\subset \btr S(A_k,L_k)$, and
therefore
$$
\iint\limits_{A_kB_kV} \left |K\right |~dS \leq \iint\limits_{\btr S(A_k,L_k)}
\left |K\right |~dS.
$$
Since the sequence $\{L_k\}$ is certainly bounded and
$\{\alpha_k\}\lra 0$, when $k\to\infty$, the areas of $\btr
S(A_k,L_k)$ will approach zero, which implies that
$$
\iint\limits_{A_kB_kV} K~dS \stackrel{k\to\infty}{\lra} 0.
$$
Now one can use formula (\ref{abo}) to deduce that $\mu=\pi$, which
will contradict the assumption that $\mu<\pi$ in the triangle $\btr
ABV$. A similar argument for  $\{A_k\}$ also  results in a
contradiction.
\end{proof}

\parskip=1mm

\noindent Siena College, School of Science\\
515 Loudon Road, Loudonville NY 12211

\noindent {\small nkrylov@siena.edu {\it and} rogers@siena.edu}

\end{document}